\documentclass[sn-mathphys-num]{sn-jnl}

\usepackage{graphicx}%
\usepackage{multirow}%
\usepackage{amsmath,amssymb,amsfonts}%
\usepackage{amsthm}%
\usepackage{mathrsfs}%
\usepackage[title]{appendix}%
\usepackage{xcolor}%
\usepackage{textcomp}%
\usepackage{manyfoot}%
\usepackage{booktabs}%
\usepackage{algorithm}%
\usepackage{algorithmicx}%
\usepackage{algpseudocode}%
\usepackage{listings}%

\usepackage{epsfig,amsfonts}
\usepackage{amsmath}
\usepackage{amsthm}
\usepackage{graphicx,color}
\usepackage{pgf}
\usepackage{url}
\usepackage{algorithm,algpseudocode}
\usepackage{lineno}
\usepackage{mathtools}
\usepackage{adjustbox}
\usepackage{setspace} 

\theoremstyle{thmstyletwo}%

\theoremstyle{thmstylethree}%

\newtheorem{lemma}{Lemma}

\counterwithin{equation}{section}

\newcommand\Tstrut{\rule{0pt}{2.6ex}}       

\newcommand{\resetcounters}{\setcounter{equation}{0} \setcounter{figure}{0}
 \setcounter{table}{0}}

\begin{document}

\title[Incomplete Cholesky factorizations in half precision arithmetic]{Developing robust incomplete Cholesky factorizations in half precision arithmetic}


\author*[1,2]{\fnm{Jennifer} \sur{Scott}}\email{jennifer.scott@reading.ac.uk}
\equalcont{These authors contributed equally to this work.}

\author[3]{\fnm{Miroslav} \sur{T\r{u}ma}}\email{mirektuma@karlin.mff.cuni.cz}
\equalcont{These authors contributed equally to this work.}

\affil*[1]{\orgdiv{School of Mathematical, Physical and Computational Sciences}, \orgname{University of Reading}, 
\orgaddress{\city{Reading RG6 6AQ},  \country{UK}}}

\affil[2]{\orgname{STFC Rutherford Appleton Laboratory}, \orgaddress{\city{Didcot, Oxfordshire,\\ OX11 0QX}, \country{UK}.} ORCID iD: 0000-0003-2130-1091}

\affil[3]{\orgdiv{Department of Numerical Mathematics, Faculty of Mathematics and Physics}, 
\orgname{Charles University}, \orgaddress{\country{Czech Republic}.} ORCID iD: 0000-0003-2808-6929}


\abstract{Incomplete factorizations have long been popular  general-purpose algebraic
preconditioners for solving large sparse  linear systems
of equations. Guaranteeing the factorization is
breakdown free while computing a high quality preconditioner is challenging.
A resurgence of interest in using low precision arithmetic
makes the search for robustness more important and more challenging.
In this paper, we focus on ill-conditioned
symmetric positive definite problems and explore a number of approaches for preventing and handling breakdowns: prescaling of the system matrix, a look-ahead strategy to anticipate breakdown as early as possible, the
use of global shifts, and a modification of an idea developed in 
the field of numerical optimization for the complete Cholesky factorization of dense matrices. Our numerical simulations target highly ill-conditioned sparse linear systems 
with the goal of computing the factors in half precision arithmetic and
then achieving double precision accuracy using mixed precision refinement. We also consider the often overlooked issue of growth in the
sizes of entries in the factors that can occur when using
any precision and can render the computed factors
ineffective as preconditioners.
}




\maketitle

\section{Introduction}

Our interest is in solving large-scale symmetric
positive definite (SPD) linear systems of equations $A x = b$.
Incomplete Cholesky (IC) factorizations 
of the form $A \approx LL^T$, where 
the factor $L$ is a sparse lower triangular matrix,
have long been important and well-used algebraic preconditioners for use with iterative solvers.
While they are general purpose and,
when compared to sparse direct solvers, require modest computational
resources,  they do have drawbacks.
Their effectiveness can be highly application dependent and, although significant effort has gone 
into developing a strong theoretical background, most results are limited
to model problems. To be useful in practice, their computation and application must be efficient and robust.

Traditionally, matrix factorizations have most often been computed using double precision floating-point
arithmetic, which nowadays corresponds to a 64-bit floating-point number format.
However, half precision  arithmetic is being increasingly supported by modern hardware
and because it can offer speed benefits while using less energy and memory,
there has been significant interest in recent years in its
use in numerical linear algebra; see the comprehensive review~\cite{hima:2022} and references therein.
For linear systems, one strategy that has received attention is
GMRES-IR  \cite{cahi:2017}. The idea is to compute
the matrix factors in low precision arithmetic
and then employ them as preconditioners for  GMRES
within mixed precision iterative refinement (see also \cite{ames:2024,cahi:2018,hima:2022}). For SPD systems, 
GMRES can potentially be replaced by the CG (conjugate gradient) method \cite{hipr:2021}.
Using low precision incomplete factors may enable much larger problems to be solved
(normally at the cost of more iterations to achieve the requested accuracy). 
In an initial study \cite{sctu:2024a}, we explored this approach, focusing on the safe avoidance of overflows that 
can occur when computing matrix factors using low precision arithmetic.

When using any precision, breakdown can occur during an incomplete Cholesky factorization
of a general SPD matrix, that is, a pivot (diagonal entry of a Schur complement)
may be zero or negative where an exact Cholesky factorization would
have only positive pivots, or a computation within the
factorization may overflow. A review is given in the 
Lapack Working Note \cite{eijk:99}.
Breakdown is more likely when using low precision arithmetic.
This is partly because  the 
initial matrix $A$ must be ``squeezed'' into half precision, which may mean the resulting matrix
is not (sufficiently) positive definite for the factorization to
be successful \cite{hipr:2019}.
But, in addition, overflows
outside the narrow range of possible numerical values are an ever-present danger.
Our interest lies in exploring strategies to prevent breakdown during sparse matrix
factorizations. Importantly, these must be robust, inexpensive
and not cause serious degradation to the preconditioner quality.

Building on our earlier work
on breakdowns within half precision sparse matrix incomplete factorizations
\cite{sctu:2024a},
this paper makes the following contributions. Firstly,
we consider and compare the performance in half precision arithmetic of 
a number of strategies to limit the likelihood of breakdown: (a) prescaling the matrix before it is squeezed, (b) look-ahead that checks
the diagonal entries
of the partially factorized matrix at each major  step of the factorization, (c) the use of
global shifts, and (d) an approach based on locally modifying the factorization.
The latter was originally used to modify approximate (dense) Hessian matrices in the field of numerical optimization.
Here we seek to apply it to sparse problems, in combination with low precision arithmetic, as a strategy that avoids the restarting needed with the use of global shifts. Secondly, we demonstrate that, even if double precision is
used throughout the computation and breakdown
does not occur, it is essential to 
take action to prevent growth in the factor entries because otherwise,
the factors can be ineffective as preconditioners.
Finally, we concentrate our numerical experiments on highly ill-conditioned linear systems and the challenge of recovering double accuracy in the computed solution using a preconditioner computed in low precision arithmetic. We develop Fortran software that enables
us to illustrate the potential for exploiting half precision within robust approaches for tackling large-scale sparse systems.

The rest of the paper is organised as follows. 
Section~\ref{sec:basic} looks at the different stages at which
breakdown can occur within an incomplete factorization.
In Section~\ref{sec:preventing breakdown}, we present
a number of ways to prevent and handle breakdown.
Numerical results for a range of highly ill-conditioned
linear systems coming from practical applications
are presented in Section~\ref{sec:results}.
Finally, in Section~\ref{sec:conclusions}, our findings and
conclusions are summarised.

\medskip\noindent
{\bf Terminology.} 
We use high precision (denoted by fp64)
to refer to IEEE double precision (64-bit) 
 and low precision (denoted by fp16) 
 for the 1985 IEEE standard 754 half precision (16-bit).  
Note that bfloat16 is another form of half precision arithmetic.
It has 8 bits in the significand and (as in fp32) 8 bits in the exponent. We do not use it in this paper because our software is written in Fortran and,
as far as we are aware, there are currently no Fortran compilers that support the use of bfloat16. 
Table~\ref{table:precisions} summarises  the parameters for the precisions used in this paper. 

\begin{table}[htbp]
\caption{Parameters for fp16, fp32, and fp64 arithmetic: the number of bits in the significand (including the implicit most significant bit) and
exponent, unit roundoff $u$, smallest positive (subnormal) number $x^s_{min}$ , smallest normalized positive
number $x_{min}$, and largest finite number $x_{max}$, all given to three significant figures.
}
 \label{table:precisions}
\vspace{3mm}
\footnotesize
\small
\begin{tabular}{cllllll}
\hline\Tstrut
&
Signif. &
 Exp. &
$ \;\;u $ &
 $\;x^s_{min}$  &
 $\;x_{min}$  &
 $\;x_{max}$  \\
 \hline\Tstrut
fp16  & 11 & 5 & $4.88 \times 10^{-4}$ & $5.96 \times 10^{-8}$ & $6.10 \times 10^{-5}$ & $6.55 \times 10^4$ \\
fp32 & 24 & 8 & $5.96 \times 10^{-8}$ & $1.40 \times 10^{-45}$ & $1.18 \times 10^{-38}$ & $3.40 \times 10^{38}$ \\
fp64 & 53 & 11 & $1.11 \times 10^{-16}$ &$4.94 \times 10^{-324}$ & $2.22 \times 10^{-308}$ & $1.80 \times 10^{308}$ \\

\hline
\end{tabular}
\end{table}

\section{Possible breakdowns within IC factorizations}\label{sec:basic}

A myriad of approaches for computing incomplete factorizations of sparse matrices
have been developed, modified and refined over many years. 
Some combine the factorization with an initial
step that discards small
entries in $A$ (sparsification).
For  details of possible variants, we recommend 
\cite{chvd:97,saad:03}, while a comprehensive discussion of early strategies can be found in \cite{illi:92}; see also
\cite{sctu:2011} for a short history and the recent monograph \cite{sctu:2023a} for a broad overview and skeleton algorithms.
We note that significant progress in the field of preconditioning has been achieved by looking for incomplete factorizations that are breakdown-free because of the properties of the matrix $A$, for example, for M and H-matrices.  A seminal paper on this is \cite{meva:77}; see also the summary in \cite{meur:99}.

Algorithm~\ref{alg:ic_generic} outlines a basic (right-looking) incomplete
Cholesky (IC) factorization of a sparse SPD matrix $A= \{a_{ij}\}$\footnote{The Algorithm can be modified to compute a square-root free LDLT factorization in which 
$L$ has unit diagonal entries and $D$ has positive entries.}. 
It assumes a target sparsity pattern ${\mathcal S}\{L\}$ for 
the incomplete factor $L=\{l_{ij}\}$ is provided, where
\begin{equation*}
{\mathcal S}\{L\}  = \{ (i,j) \, | \ l_{ij} \neq 0, \, 1 \le  j \le i \le n \}.
\end{equation*}
The simplest case ${\mathcal S}\{L\}  = {\mathcal S}\{A\}$
is called an $IC(0)$ factorization. 
Modifications to Algorithm~\ref{alg:ic_generic} can be made to incorporate threshold dropping strategies
and to determine ${\mathcal S}\{L\}$ as the method proceeds.
At each major step $k$,  outer product updates are applied to the part of the matrix
that has  yet to be factored (Lines 7--11).

\medskip
\begin{algorithm}\caption{Basic right-looking sparse IC factorization}
\textbf{Input:} Sparse SPD matrix $A$ and a target sparsity pattern ${\mathcal S}\{L\}$  \\
\textbf{Output:} Incomplete Cholesky factorization  $A \approx  L L^T$
\label{alg:ic_generic}

\setstretch{1.17}\begin{algorithmic}[1]
\State $l_{ij} = a_{ij}$ for all $(i,j) \in {\mathcal S}\{L\}$
\For{$k=1:n$}\Comment{Start of $k$-th major step}
  \State $ l_{kk}\leftarrow (l_{kk})^{1/2}$\Comment{Diagonal entry is the pivot}

  \For{$i \in \{i>k \, |\, (i,k) \in {\mathcal S}\{L\} \}$}
    \State    $l_{ik} \leftarrow l_{ik} / l_{kk}$ \Comment{Scale pivot column $k$ of the incomplete factor by the pivot}
  \EndFor \Comment{Column $k$ of $L$ has been computed}
  \For{$j \in \{j>k \, |\, (j,k) \in {\mathcal S}\{L\} \}$}
    \For{$i \in \{i\ge k \, |\, (i,k) \in {\mathcal S}\{L\} \}$}
     \State    $l_{ij} \leftarrow l_{ij} - l_{ik} l_{jk}$ \Comment{Update operation on column $j>k$}
   \EndFor
  \EndFor
\EndFor
\end{algorithmic}
\end{algorithm}
\medskip

Unfortunately, unlike a complete Cholesky factorization, there is no guarantee in general  that an IC algorithm 
will not break down or exhibit large growth 
in the size of the factor entries (even when using double precision arithmetic). This is illustrated by the following well-conditioned SPD matrix in which $\delta > 0$ is small
\begin{equation*}
A = \begin{pmatrix}
            3   &  -2  &  0 & 2 &  0 \cr
           -2   &   3  & -2 & c &  0    \cr
            0   &  -2  &  3 &-2 &  0    \cr
            2   &   c  & -2 & 8+2\delta &  2     \cr
            0   &   0  &  0 & 2 &  8     \cr
\end{pmatrix}.
\end{equation*}
Choosing $\delta \ll 1$ 
and $c = 1$ results in no growth in the entries of the $IC(0)$ factor and no breakdown.
However, if $c=0$ and entries (2,4) and (4,2) are removed from ${\mathcal S}\{A\}$ then the $IC(0)$ factor becomes 
\begin{equation*}
\begin{pmatrix}
        d_1   &         \cr
     -2/d_1    &        d_2  &     \cr
        0    &     -2/d_2   &    d_3  &   \cr
      2/d_1    &        0   &   -2/d_3 &   d_4 &   \cr
        0    &        0   &    0    & 2/d_4 &  d_5 \cr
\end{pmatrix}, 
\end{equation*}
with $d_1^2 = 3$, $d_2^2 = 5/3$, $d_3^2 = 3/5$, $d_4^2 = 2\delta$, and $d_5^2 = 8-2/\delta$.
In this case, if $\delta \ll 1$ then there is large growth in the (5,4) entry and the factorization breaks down 
because the (5,5) entry is negative (for any working precision). 

There are three places in Algorithm~\ref{alg:ic_generic} where breakdown
can occur. Following \cite{sctu:2024a}, we refer to
these as  B1, B2, and B3 breakdowns.
\begin{itemize}
\item B1: The diagonal entry $l_{kk}$ may be unacceptably small or negative. 
\item B2: The column scaling $l_{ik} \leftarrow l_{ik} / l_{kk}$ may overflow.
\item B3: The update operation $l_{ij} \leftarrow l_{ij} - l_{ik} l_{jk}$ may overflow.
\end{itemize}
To develop robust IC factorization implementations, 
breakdowns must either be avoided or 
they must  be detected and handled by restarting
the computation with revised data. We seek to avoid breakdowns
but, as we cannot guarantee there will be
no breakdowns, we still need to monitor for them. A recent
study involving multi-precision iterative refinement
used  the functions offered by MATLAB to
check the computed factors for Inf and/or NaN  entries
and took action if such entries were found \cite{okca:2022}. This is not a
practical procedure for general use.
One possible strategy 
is to use IEEE-754 floating-point exception handling. This allows overflows 
to occur, the execution continues until
a status flag is checked and, at this point,
if overflow has been detected, restarting is initiated;
see, for example, \cite{deli:94}. This is straightforward
but requires the user to employ the correct
compiler flags, which may be challenging, for instance, when a solver
is interfaced from other languages. 
Furthermore, it is likely that non-IEEE arithmetics will
gain traction in the future \cite{lilh:2018}. An alternative 
and potentially more flexible strategy is 
to incorporate explicit tests for breakdown
into the factorization algorithm. In this case, for an implementation
to be robust, the tests employed
must only use operations that cannot themselves overflow.

An operation is said to be {\it safe} in the precision being used
if it cannot overflow. To safely detect B1 breakdown it is sufficient
to check at Line 3 of Algorithm~\ref{alg:ic_generic}
that $l_{kk} \ge \tau_u $, where the  threshold parameter satisfies $\tau_u>1/x_{max}$. This ensures $(l_{kk})^{-1}< x_{max}$. 
Typical values are $\tau_u = 10^{-5}$ for 
half precision factorizations and  $\tau_u = 10^{-20}$ for double precision \cite{sctu:2024a};    these 
are used in our reported experiments (Section~\ref{sec:results}).
 B2 breakdown can happen at Line 5. Let $l_{kmax}$
denote the entry below the diagonal  in column $k$ of largest absolute value, that is,
\begin{equation}\label{eq:lmax}
l_{kmax} = \max_{i>k}\, \{ |l_{ik}| : (i,k) \in {\mathcal S}\{L\}  \}.
\end{equation}
If $l_{kmax} \le x_{max}$ and $1 \le l_{kk} \le x_{max}$ or
$l_{kk} \ge l_{kmax}/x_{max}$ then it is safe to compute $l_{kmax}/l_{kk}$ (and thus safe to scale column $k$).
B3 breakdown can occur at Line 9. We give an algorithm for safely detecting
B3 breakdown in \cite{sctu:2024a}.
Given scalars $a,b,c$ such that $|a|,|b|,|c| \le x_{max}$, the algorithm
returns $v=a-bc$ or a flag to indicate $v$ cannot be computed safely.
It does this in two stages: it first checks whether $w = bc$ can be computed safely and then
whether $v = a - w$ can be computed safely.

Note that although we are focusing on SPD problems and
IC factorizations, B1, B2 and B3 breakdowns are
also possible during complete or incomplete factorizations of nonsymmetric
sparse matrices. Indeed, for non SPD problems,
B2 breakdowns (that is, overflow of one or more entries when the pivot column is divided by the pivot) in particular may be more likely to occur (although remains uncommon if the matrix is well-scaled).
To demonstrate how B2 breakdown can happen, consider the following 
nonsymmetric matrix, which has some large off-diagonal entries
\begin{equation*} 
A=\begin{pmatrix}
3  & -2 & 0 & 2 & 2    \cr
-2 & 3 & -2 & 0 & 0    \cr
0 & -2 & 3 & -1 & -1   \cr
2 & 0 & -2 & 2.01 & 2.01 \cr
1000 & 1000 & 1000 & 1000 & 100 \cr
\end{pmatrix}.
\end{equation*}
 The LU factorization of $A$ is given by
 \begin{equation*}
A = LU  = 
\begin{pmatrix}
1  &     \cr
 -2/3  & 1   \cr
        0  & -1.2 & 1     \cr
  2/3   & 0.8 & -2/3  & 1    \cr
  1000/3   &  1000 & 5000  & -400000 & 1 \cr
\end{pmatrix}
\begin{pmatrix}
3 & -2 & 2 & 0 & 2 \cr
& 5/3 & -2 & 4/3 & 4/3 \cr
&& 0.6 & 0.6 & 0.6 \cr
&&& 0.01 & 0.01 \cr
&&&& -900
\end{pmatrix}.
\end{equation*}
When using fp16 arithmetic,  B2 breakdown occurs 
when performing the $4$th elimination step
because the $(5,4)$ entry overflows (-4000 is divided by 0.01).

\section{Preventing and handling breakdown in IC factorizations}
\label{sec:preventing breakdown}
While the use of safe tests allows action to be taken
before breakdown occurs or the use of IEEE exception handling
can capture breakdown, our objective is to reduce the likelihood of breakdown. This will limit the overheads involved
in handling breakdowns and the effects on the quality of the computed
factorizations through modifications to the data.
Breakdown is much more likely to happen when using low precision
arithmetic because of the greater likelihood of overflows occurring.

\subsection{Avoiding breakdown by prescaling}

For both direct and iterative methods 
for solving systems of equations it is often beneficial to
prescale the matrix, that is, to determine
diagonal matrices $S_r$ and $S_l$ (with $S_r = S_l$ in the symmetric case) such
that the scaled matrix $\widehat A = S_r^{-1} A S_l^{-1}$ 
is ``nicer'' than the original $A$. By nicer, we mean that, compared with
solving $Ax = b$, it
is easier to solve the system $\widehat A y = S_r^{-1}b$ and then set
$x = S_l^{-1} y$.
When working in fp16 arithmetic, scaling is essential because of the
narrow range of the arithmetic (recall Table~\ref{table:precisions}).
Numbers of absolute value outside the interval $[x_{min}^s , x_{max} ] = [5.96 \times 10^{-8}, 6.55 \times 10^4 ]$
cannot be represented in fp16 arithmetic and they
underflow or overflow when converted to fp16 arithmetic. Moreover, to avoid the
performance penalty of handling subnormal numbers,  in practice numbers with small absolute values
are often flushed to zero (that is, replaced by zero). 
Before factorizing $\widehat A$ in fp16 arithmetic, a scaling is chosen so that when converting (squeezing) the scaled matrix into fp16, overflow is avoided. In this initial squeezing of the matrix, we flush to zero all entries in the scaled matrix of absolute value less than $10^{-5}$. 
For incomplete factorizations, numbers that underflow or are flushed
to zero are not necessarily a concern because the factorization is approximate. However,
the resulting sparsification may mean that the scaled and squeezed matrix is close to being indefinite.

No single approach to constructing a scaling is universally the best and sparse solvers frequently
include a number of options to allow users to experiment to determine
the most effective for their applications (or to supply their own scaling). 
Our experience with IC factorizations of SPD matrices is that it is 
normally sufficient to use simple
$l_2$-norm scaling (that is, using fp64 arithmetic, we compute $S_r = S_l= D^{1/2}$, where $d_{ii}$ 
is the 2-norm of row $i$ of $A$), resulting in the absolute values
of the entries of the scaled matrix $\widehat A$ being at most 1. This is used in the current study (but see \cite{hipz:2019}, where equilibration scaling is used and \cite{hipr:2021} where scaling by the square root of
the diagonal is used).

\subsection{Preventing breakdown by incorporating look-ahead}
\label{subsec:look_ahead}
Recall that the computation of the diagonal entries of
the factor in a (complete or incomplete) Cholesky factorization
are based on 
\begin{equation*}
    l_{jj} = a_{jj} - \sum_{i < j} l_{ij}^2.
\end{equation*}
Initially, $l_{jj} = a_{jj}$ and at each stage of the factorization a positive (or zero) term
is subtracted from it so that 
$l_{jj}$ either decreases or remains the same on each major step $k$.
Thus, to detect potential B1 breakdown as early as possible, look-ahead
can be used whereby, at each step $k$, 
the remaining diagonal entries $l_{jj}$ ($j > k$)
are updated (using safe operations) and tested.
For the right-looking Algorithm~\ref{alg:ic_generic}, it is straightforward to
incorporate testing but, for some IC variants, it may be necessary to hold
a copy of the diagonal entries of the factor.
Look-ahead is employed in some well-known
fp64 arithmetic implementations of IC factorizations e.g., \cite{limo:99,sctu:2014a}.
Algorithm~\ref{alg:ic_look_ahead} is a modified
version of Algorithm~\ref{alg:ic_generic} that includes 
checks for breakdown. The safe test 
for B3 breakdown can be modified so that, ahead of the loop at
Line 13, a check is made that the entry of maximum magnitude 
in column $k$ (that is, $l_{kmax}$ from \eqref{eq:lmax}) is less than
$(x_{max})^{1/2}$. If it is not, $flag= 3$ is returned. Otherwise, 
the multiplication of $l_{ik}$ and $l_{jk}$ is safe and, at Line 17,
only the subtraction needs to be checked.
If in place of the safe tests, IEEE exception handling is used
then the IEEE overflow
flag should be tested at the end of each major loop (that is, between Lines 20 and 21).

\begin{algorithm}\caption{Right-looking IC factorization with safe checks for breakdown} 
\textbf{Input:}  SPD matrix $A$, a target sparsity pattern ${\mathcal S}\{L\}$, parameter  $\tau_u>0$ \\
\textbf{Output:} Either $flag = 0$ and  $A \approx LL^T$ or  
$flag >0 $ (breakdown detected)

\label{alg:ic_look_ahead}
\setstretch{1.17}
\begin{algorithmic}[1]
\State  $l_{ij} = a_{ij}$ for all $(i,j) \in {\mathcal S}\{L\}$ 
\State  $flag = 0$
  \State {\bf if} $l_{11} < \tau_u$ {\bf then} $flag =1$ and {\bf  return} \Comment{B1 breakdown}
\For{$k=1:n$}\Comment{Start of $k$-th major step}
  \State $ l_{kk}\leftarrow (l_{kk})^{1/2}$
  \If{ $l_{kk} \ge 1$ or $l_{kk} \ge l_{max}/x_{max}$}\Comment{$l_{max}$ is largest off-diagonal entry (\ref{eq:lmax})}
  \For{ $i \in \{i >k \, | \, (i,k) \in {\mathcal S}\{L\}\}$} 
    \State  $l_{ik} \leftarrow l_{ik} / l_{kk}$ \Comment{Perform safe scaling}
  \EndFor\Comment{Column $k$ of $L$ has been computed}
  \Else
         \State  $flag =2$ and {\bf return}\Comment{B2 breakdown}
  \EndIf
  \For{ $j \in \{j >k \, | \, (j,k) \in {\mathcal S}\{L\}\}$} \Comment{Update columns $j>k$ of $L$}
    \For{$i \in \{i \ge j \, | \, (i,j) \in {\mathcal S}\{L\}\}$} 
      \State Test entry $(i,j)$ can be updated safely \Comment{Use Algorithm 2.2 of \cite{sctu:2024a}}
      \State {\bf if} { not safe to update}  {\bf then} $flag =3$ and {\bf  return} \Comment{B3 breakdown}
     \State    $l_{ij} \leftarrow l_{ij} - l_{ik} l_{jk}$\Comment{Perform safe update operation}
   \EndFor
    \State {\bf if}{ $l_{ii} < \tau_u$}  {\bf then} $flag =1$ and {\bf  return}
     \Comment{B1 breakdown}
  \EndFor
\EndFor
\end{algorithmic}
\end{algorithm}
\medskip

To see the usefulness of look-ahead, consider the following matrix 
\begin{equation*}
    A = \begin{pmatrix}
        3 & -2 & 0 & 1 & 2  \cr
        -2 & 3 & -2 & 0 & 0  \cr
        0 & -2 & 3 & 0 & -2  \cr
        1 & 0 & 0 & 5 & 0  \cr
        2 & 0 & -2 & 0 & 8   \cr
    \end{pmatrix}.
\end{equation*}
It is easy to check that $A$ is SPD with condition number $\kappa_2(A) \approx 2 \times 10^6$. 
If the $IC(0)$ factorization is computed in exact
arithmetic then the entry $(5,5)$  of $L$
is zero (B1 breakdown).
The look-ahead strategy reveals this at the third step,
thus reducing the work performed before breakdown is detected.

A consequence of look-ahead 
is that, through the early detection of
B1 breakdowns and taking action to
prevent such breakdowns,  B3  breakdowns are indirectly prevented.
In our numerical experiments on problems coming
from real applications, all breakdowns when using
fp16 arithmetic were of type B1 when look-ahead was incorporated. However,
B2 and B3 breakdowns remain possible. 
Consider the following well-conditioned SPD matrix
and its complete Cholesky factor (three decimal places):
\begin{equation*} 
A = \begin{pmatrix}
            3   &  -2  &  0 & 2 &  0 \cr
           -2   &   3  & -2 & 0 &  0    \cr
            0   &  -2  &  3 &-2 &  0    \cr
            2   &   0  & -2 & 8.00007 &  550     \cr
            0   &   0  &  0 & 550 &  60000     \cr
            \end{pmatrix}, \qquad
L = 
\begin{pmatrix}
1.732  &   \cr
 -1.155  & 1.291 &  \cr
        0  & -1.549 & 0.775 &    \cr
  1.155   & 1.033 &  -0.516 & 2.309 &    \cr
     0     & 0 &        0      & 95.254 & 30.414 \cr
\end{pmatrix}.
\end{equation*}
Observe that the $(4,2)$ entry has filled in. The $IC(0)$ 
factorization does not allow fill-in and, after four steps, the first four
columns of the  $IC(0)$ factor of $A$ are given by
\begin{equation*} 
L_{1:5,1:4} =
\begin{pmatrix}
1.732  &     \cr
 -1.155  & 1.291 &   \cr
        0  & -1.549 & 0.777 &      \cr
  1.155   & 0       &  -2.582 & 0.008    \cr
     0     & 0 &        0      & 65738  \cr
\end{pmatrix}.
\end{equation*}
In fp16 arithmetic, the  $(5,4)$ entry overflows (B3 breakdown).
This happens even with look-ahead because the first three
entries in row 5 of $A$ are zero and so the
$(5,4)$ entry is not updated until after column 4 has been computed.
In this example, the 
$(4,4)$ entry before its square root is taken  is greater
than $\tau_u$ (it is equal to $7 \times 10^{-5}$) and so there
is no B1 breakdown in column 4.

\subsection{Global shifting to handle breakdown}\label{subsec:global}

Once potential breakdown has been detected (either through the B1-B3 tests
or using IEEE exception handling)
 and the
factorization halted, a common approach is to modify all the
diagonal entries  by selecting $\alpha>0$,
replacing the scaled matrix $\widehat A$ by $\widehat A + \alpha I$ 
and restarting the factorization.
In exact arithmetic, there is always an $\alpha^\ast$ such that for all $\alpha \ge \alpha^\ast$
the IC factorization of $\widehat A + \alpha I$ exists \cite{mant:80}. 
In practice, $\alpha^\ast$ is unlikely to be known a priori and it may
be necessary to restart the factorization  a number of times with ever
larger shifts. If $\alpha$ was to reach $x_{max}$ in the working precision $u$ then there would have been large growth in the factor entries and it would be necessary to restart using higher precision.
However, this would be unlikely to result in a useful preconditioner and $\alpha$ growing in this way was not observed in any of our tests (including highly ill-conditioned examples).
Algorithm~\ref{alg:trial_error}
summarizes the global shifting strategy for a SPD matrix held in precision $u$
and for which an IC factorization in precision $u_l \ge u$ is wanted.
After each unsuccessful factorization attempt, the shift is doubled \cite{limo:99}.
Here $A_l= fl_l(\widehat A)$ denotes converting the matrix $\widehat A$ from precision $u$ to precision $u_l$. 
In practice, it is unnecessary to explicitly hold $A_l$.
Instead, entries of $\widehat A$ are cast to precision $u_l$ on the fly as needed.

Based on our experience with a range of problems,
in our reported tests, the initial shift is taken to be $\alpha_S = 10^{-3}$. 
The precise choice of the shift is not critical but it should not be unnecessarily large as this may result
in the computed factors providing poor quality preconditioners. 
Note that more sophisticated strategies for changing the shift,
which also allow the possibility for a shift to be decreased, are possible \cite{sctu:2014a}. These have been developed
for fp64 arithmetic. Our initial experiments suggest it is less clear that there are significant benefits of doing this
when using fp16 arithmetic so in our experiments we only report results for
using the simple shift doubling strategy.

\medskip
\begin{algorithm}\caption{Shifted incomplete IC factorization in precision $u_l$}
\textbf{Input:}  SPD matrix $A$ in precision $u$, diagonal scaling matrix $S$,
a target sparsity pattern ${\mathcal S}\{L\}$,
and initial shift $\alpha_S >0$ \\
\textbf{Output:} Shift $\alpha \ge 0$ and incomplete  Cholesky factorization $S^{-1}AS^{-1}+ \alpha I \approx L L^T$
in precision $u_l$.

\setstretch{1.17}\begin{algorithmic}[1]
\State $\widehat A = S^{-1}AS^{-1}$\Comment{Symmetrically scale $A$}
\State $A_l = fl_l(\widehat A)$\Comment{Convert to precision $u_l$}
\State $\alpha_0 = 0$
\For{$k = 0,1,2, \ldots$}
 \State $ A_l+\alpha_k I\approx L L^T$ in precision $u_l$ \Comment{We use Algorithm~\ref{alg:ic_generic2}}
 \State If successful then set $\alpha = \alpha_k$ and {\bf return}
 \State $\alpha_{k+1} =   \max(2 \alpha_k, \ \alpha_S)$
 \EndFor
\end{algorithmic}
\label{alg:trial_error}
\end{algorithm}
\medskip

\subsection{Local modifications to prevent breakdown}\label{subsec:modified}

The next strategy is based on seeking to guarantee 
that the factorization exists by 
bounding the off-diagonal entries in $L$.
So-called modified Cholesky factorization schemes have been widely used in nonlinear optimization to
compute Newton-like directions. Given a symmetric (and possibly indefinite) $A$,
a modified Cholesky algorithm factorizes
$A+A_E$, where $A_E$ is termed the correction matrix. The objectives
are to compute the correction at minimal additional cost and to ensure
$A+A_E$ is SPD and well-conditioned and close to $A$.
A stable approach for dense matrices was originally proposed by Gill and Murray~\cite{gimu:74}
and was subsequently refined and used by Gill, Murray and Wright (GMW)~\cite{gimw:81} and others  \cite{essc:91,cogt:96,faol:08,sces:90,sces:99}.
 
Our GMW variant allows for sparse $A$ and incomplete factorizations. In particular, the incompleteness that can lead to significant growth in the factor is a feature that implies modification of the GMW strategy is necessary. Consider the second example in Section~\ref{subsec:look_ahead}. It clearly shows that once a diagonal entry is small, but not smaller than $\tau_u$, it may be difficult to get a useful factorization by increasing this diagonal entry by an initially prescribed value that is independent of other entries in its column.
In the GMW approach, at the start of major step $k$ of the factorization algorithm, the updated diagonal
entry $l_{kk}$ is checked (before its square root is taken). If it is too small
compared to the off-diagonal entries in its column then it is modified; a parameter $\beta>0$
controls the local modification. Specifically, at step $k$, we set
\begin{equation}\label{eq:modification}
   l_{kk} = \max{ \left\{l_{kk}, \ \left(\frac{{l_{kmax}}}{\beta}\right)^2\right\}}, 
\end{equation}
where {$l_{kmax}$} is given by (\ref{eq:lmax}).
If $(l_{kmax}/{\beta})^2$
overflows (this can be safely checked) then the diagonal entry cannot be modified in this way. We call this a B4 breakdown.  If despite the local modification
potential breakdown is detected then the factorization is terminated and restarted using a global shift.
As the following result shows, the GMW$(\beta)$ strategy limits the size of the off-diagonal entries in $L$ and, 
for $\beta$ sufficiently small, it prevents B3 breakdown. 

\smallskip
\begin{lemma}\label{th:nobreak}
Let the matrix $A$ be sparse and SPD.
Assume that,  using the GMW($\beta$) strategy,
columns 1 to $j-1$ columns of the IC factor $L$ have been successfully computed 
in fp16 arithmetic. 
For  $i \ge j$ let $nz(i)$ denote
the number of nonzero entries in $L_{i,1:j-1}$.  If 
\begin{equation}\label{eq:assumption}
|a_{ij}| + \min(nz(i), nz(j)) \beta^2 \le x_{max} \quad \mbox{for all} \quad (i,j) \in {\mathcal S}\{L\},
\end{equation}
where $x_{max}$ is the largest finite number represented in fp16, 
then B3 breakdown cannot occur in the $j$-th step.  
\end{lemma}

\begin{proof}

From (\ref{eq:modification}), the off-diagonal entries  in the first $j-1$ columns of $L$ satisfy
\begin{equation*}
\frac{|l_{ik}|}{(l_{kk})^{1/2}} \le \frac{|l_{ik}| \beta}{ l_{kmax}} \le \beta, \quad 1 \le k \le j-1, \ i > k.
\end{equation*}
For any $i > j$ we have
\begin{equation}\label{eq:compute_l}
    l_{ij} = \frac{1}{(l_{jj})^{1/2}}\left(a_{ij} - \sum_{k=1}^{j-1} l_{ik}l_{jk}\right) := \frac{\widetilde l_{ij}}{(l_{jj})^{1/2}} .
\end{equation}
To avoid breakdown, we require $|\widetilde l_{ij}| \le x_{max}$. From \eqref{eq:assumption} and \eqref{eq:compute_l},
\begin{equation*}
 |\widetilde l_{ij}|   \le |a_{ij}| + 
\min(nz(i), nz(j)) \beta^2 \le x_{max} , 
\end{equation*}
 and hence B3 breakdown does not occur. 

\end{proof}

Rules to determine the parameter $\beta$ for the complete Cholesky factorization  of dense matrices (in double precision) are discussed in
\cite{gimw:81,gimu:74} but for sparse incomplete factorizations 
in fp16 arithmetic such sophisticated rules are not applicable. 
Provided $\beta$ is not very small, its value is not critical
to the quality of the preconditioner.
Given $\beta$, Algorithm~\ref{alg:ic_generic2} incorporates the  use of the
GMW($\beta$) strategy within the incomplete factorization. 
Note that the local modifications are not combined with look-ahead.
As in Algorithm~\ref{alg:ic_look_ahead}, the cost of checking for breakdowns is small. The most expensive step is computing $l_{max}$.

\begin{algorithm}\caption{Right-looking IC factorization with safe checks for breakdown and GMW local modifications} 
\textbf{Input:}  SPD matrix $A$, a target sparsity pattern ${\mathcal S}\{L\}$, parameters $\tau_u>0$ and $\beta > 0$ \\
\textbf{Output:} Either $flag = 0$ and  $A \approx LL^T$ or  
$flag >0$ (breakdown detected)

\label{alg:ic_generic2}
\setstretch{1.17}
\begin{algorithmic}[1]
\State  $l_{ij} = a_{ij}$ for all $(i,j) \in {\mathcal S}\{L\}$ 
\State Set $flag = 0$
\For{$k=1:n$}\Comment{Start of $k$-th major step}
  \If{${(l_{max}}/{\beta)^2}$ does not overflow}\Comment{$l_{max}$ is largest off-diagonal entry (\ref{eq:lmax})}
  \State  Set $l_{kk} = \max{ \left\{l_{kk}, \ (l_{max}/\beta)^2 \right\}}$. 
  \Else
         \State  $flag =4$ and {\bf return} \Comment{B4 breakdown}
  \EndIf
  \State {\bf if} $l_{kk} < \tau_u$ {\bf then} $flag =1$ and {\bf  return} \Comment{B1 breakdown}
  \State Follow Lines 5--22 of Algorithm~\ref{alg:ic_look_ahead}, with Lines 18--20 (look-ahead) removed.
\EndFor
\end{algorithmic}
\end{algorithm}
\medskip


\subsection{Recovering double precision accuracy}

Having computed an incomplete factorization in low precision, we
seek to recover (close to) double precision accuracy in the final solution
(although in many applications much less accuracy may be sufficient and 
may be all that is justified by the accuracy in the data).
In their work on using mixed precision for solving 
general linear systems, Carson and Higham~\cite{cahi:2017} introduce a variant of
iterative refinement that uses GMRES preconditioned
by the low precision LU factors of the matrix to solve the correction equation (GMRES-IR). 
Carson and Higham employ two precisions.
This was later extended to three precisions and 
then to five precisions \cite{ames:2024,cahi:2018}; 
see Algorithm~\ref{alg:gmres-ir-ic}, where we use a generic Krylov solver and the low precision incomplete factors.
Here $u$ is the working precision.
In the three-precision variant \cite{cahi:2018}, $u_p = u$, $u_p = u^2$ and
$u_l \ge u \ge u_r$ and typical combinations include $(u_l, u, u_r) =(u_{16},u_{32},u_{64})$
or $(u_{16},u_{64},u_{64})$.
Section 3.4 of \cite{ames:2024} discusses meaningful combinations of the five precisions.
They must satisfy $u^2 \le u_r \le u \le u_l$, $u_p \le u_g$ and $u_p < u_l$. The use of more than two precisions has the potential
to solve problems that are less well conditioned in less time and using less memory.

In the SPD case, a natural choice is to select the
conjugate gradient (CG) method to be the Krylov solver. The supporting rounding
error analysis for GMRES-IR relies on the backward stability of 
GMRES and preconditioned CG is not guaranteed to be backward stable \cite{gree:97a}.
Nevertheless, mixed precision results presented in \cite{hipr:2021} 
(using MATLAB code and relatively small test examples) suggest 
that in practice CG-IR can perform as well as GMRES-IR.
In our earlier paper \cite{sctu:2024a}, we experiment with IC-CG-IR and
compare it with IC-GMRES-IR using fp16 and fp64 arithemtic.
While there is little to choose between them when run on well-conditioned SPD problems, for highly ill-conditioned examples, IC-CG-IR often (but not always) requires a greater number of iterations to obtain double precision accuracy.

\medskip
\begin{algorithm}\caption{IC-Krylov-IR: Krylov solver iterative refinement using five precisions}
\textbf{Input:}  SPD matrix $A$ and vector $b$ in precision $u$,   five precisions $u_r$,
$u_g$, $u_p$, $u$ and $u_{\ell}$, maximum number of outer iterations
$itmax>0$\\
\textbf{Output:} Computed solution of the system $Ax = b$ in precision $u$
\label{alg:gmres-ir-ic}

\setstretch{1.17}\begin{algorithmic}[1]
\State  Compute IC factorization of $A$ in precision $u_{\ell}$
\State  Initialize $x_1 = 0$
\For{$ i = 1 : itmax$ or until convergence}
\State  Compute $r_i = b - Ax_i$ in precision $u_r$; store $r_i$ in precision $u$
\State  Use a preconditioned Krylov solver to solve $Ad_i = r_i$ at precision
$u_g$, with
 preconditioning and products with $A$ in precision $u_p$; store $d_i$ in precision $u$
\State  Compute $x_{i+1} = x_i + d_i$ in precision $u$
\EndFor
\end{algorithmic}
\end{algorithm}
\medskip

\section{Numerical experiments}

\label{sec:results}

We follow a number of others working on the development
of numerical linear algebra algorithms in mixed precision in
performing experiments that aim  to explore the 
feasibility of the ideas by using half precision (see, for example,  \cite{hipr:2019,hipr:2021,cahp:2020a,cakh:2023}).
We want the option to experiment with sparse problems
that may be too large for MATLAB and have 
chosen to develop our software in Fortran.
We use the NAG compiler (Version 7.1, Build 7118). As far as we know, it
is the only multi-platform Fortran compiler that currently fully supports the use of fp16 arithmetic 
and conforms to the IEEE standard. In addition,
using the {\tt -roundhreal} option, all half-precision operations are 
rounded to half precision, both at compile time and runtime.
Our numerical experiments are performed on a Windows 11-Pro-based machine with 
an Intel(R) Core(TM) i5-10505 CPU processor (3.20 GHz).

Our test set of SPD matrices is given in Table~\ref{T:test problems}.
This  set was used in our earlier study \cite{sctu:2024a}.
For consistency with that study, we do not reorder the matrix $A$.
The problems come from a variety of application areas and are of
different sizes and densities. 
As we expect that successfully using fp16 arithmetic will be most challenging 
for ill-conditioned problems, the problems were chosen because
they all have a large  estimated condition
number (in the range $10^7 - 10^{16}$).
Many are initially poorly scaled and some contain entries that overflow
in fp16 and thus prescaling of $A$ is essential.
The right-hand side vector $b$ is constructed by setting
the solution $x$ to be the vector of 1's.
In Table~\ref{T:test problems}, we also report the number of entries in 
the ``scaled and squeezed'' matrix $A_l$ (see Line 2 of Algorithm~\ref{alg:trial_error}).
The squeezing discards
all entries of the scaled matrix with absolute value less than $\tau_u=10^{-5}$. We see that this can lead to the loss
of a significant number of entries.
\resetcounters
\begin{table}[htbp]
\caption{Statistics for our ill-conditioned test examples. 
$nnz(A)$ denotes the number of entries in the lower triangular part of $A$.
$normA$ and $normb$ are the infinity norms of $A$ and $b$. $cond2$
is a computed estimate of the condition number of $A$ in the 2-norm.
$nnz(A_l)$ is the number of entries in the lower triangular part of
the matrix after scaling and squeezing.
}
\label{T:test problems}\vspace{3mm}
{
\footnotesize
\begin{tabular}{lrrlllr} \hline
{Identifier} &
 {$n$} &
 {$nnz(A)$} &
 {$normA$} &
 {$normb$} &
 {$cond2$} &
  {$nnz(A_l)$} \\ 
\hline\Tstrut
Boeing/msc01050          &  1050 &   1.51$\times 10^4$ &   2.58$\times 10^7$    &   1.90$\times 10^6$     &  4.58$\times 10^{15}$ & 4.63$\times 10^3$ \\ 
HB/bcsstk11              &  1473 &   1.79$\times 10^4$ &   1.21$\times 10^{10}$ &   7.05$\times 10^8$     &  2.21$\times 10^8$    & 6.73$\times 10^3$ \\ 
HB/bcsstk26              &  1922 &   1.61$\times 10^4$ &   1.68$\times 10^{11}$ &   8.99$\times 10^{10}$  &  1.66$\times 10^8$    & 6.59$\times 10^3$ \\ 
HB/bcsstk24              &  3562 &   8.17$\times 10^4$ &   5.28$\times 10^{14}$ &   4.21$\times 10^{13}$  &  1.95$\times 10^{11}$ & 3.89$\times 10^4$ \\ 
HB/bcsstk16              &  4884 &   1.48$\times 10^5$ &   4.12$\times 10^{10}$ &   9.22$\times 10^8$     &  4.94$\times 10^9$    & 5.24$\times 10^4$ \\ 
Cylshell/s2rmt3m1        &  5489 &   1.13$\times 10^5$ &   9.84$\times 10^5$    &   1.73$\times 10^4$     &  2.50$\times 10^8$    & 5.09$\times 10^4$ \\ 
Cylshell/s3rmt3m1        &  5489 &   1.13$\times 10^5$ &   1.01$\times 10^5$    &   1.73$\times 10^3$     &  2.48$\times 10^{10}$ & 5.07$\times 10^4$ \\ 
Boeing/bcsstk38          &  8032 &   1.82$\times 10^5$ &   4.50$\times 10^{11}$ &   4.04$\times 10^{11}$  &  5.52$\times 10^{16}$ & 7.83$\times 10^4$ \\ 
Boeing/msc10848          & 10848 &   6.20$\times 10^5$ &   4.58$\times 10^{13}$ &   6.19$\times 10^{11}$  &  9.97$\times 10^9$    & 3.02$\times 10^5$ \\ 
Oberwolfach/t2dah$\_$e   & 11445 &   9.38$\times 10^4$ &   2.20$\times 10^{-5}$ &   1.40$\times 10^{-5}$  &  7.23$\times 10^8$    & 4.88$\times 10^4$ \\ 
Boeing/ct20stif          & 52329 &   1.38$\times 10^6$ &   8.99$\times 10^{11}$ &   8.87$\times 10^{11}$  &  1.18$\times 10^{12}$ & 6.30$\times 10^5$ \\ 
DNVS/shipsec8            &114919 &   3.38$\times 10^6$ &   7.31$\times 10^{12}$ &   4.15$\times 10^{11}$  &  2.40$\times 10^{13}$ & 7.70$\times 10^5$ \\ 
GHS$\_$psdef/hood        &220542 &   5.49$\times 10^6$ &   2.23$\times 10^9$    &   1.51$\times 10^8$     &  5.35$\times 10^7$    & 2.66$\times 10^6$ \\ 
Um/offshore              &259789 &   2.25$\times 10^6$ &   1.44$\times 10^{15}$ &   1.16$\times 10^{15}$  &  4.26$\times 10^9$    & 1.17$\times 10^6$ \\ 
\hline
\end{tabular}
}
\end{table}

Our results are for the level-based incomplete Cholesky factorization preconditioner $IC(\ell)$ with $\ell = 2$ and 3  \cite{watt:81}.
The number of entries in the incomplete factor $L$ increases
as the parameter $\ell$ increases. $IC(0)$ is a very simple preconditioner
in which $L$ has the same sparsity pattern as $A$. In practice, using
very small $\ell$ may be sufficient for solving well-conditioned problems but, as shown in \cite{sctu:2024a}, the resulting preconditioner is
often not powerful enough
to successfully tackle ill-conditioned examples (particularly when computed using fp16 arithmetic).
We refer to the $IC(\ell)$ factorizations computed using half and double precision arithmetic as
fp16-$IC(\ell)$ and fp64-$IC(\ell)$, respectively.
The key difference between the two versions is
that for the former, during the incomplete factorization,
we incorporate the safe checks for the scaling and update operations; for the fp64 version,
tests for  B1 breakdown are performed (B2 and B3 breakdowns
were not encountered in our double precision experiments). 
The solves with $L$ and $L^T$ employ the $L$
factor in double precision. 
This can be done by casting the data into double precision 
and making an explicit copy of $L$; this negates the important benefit that half precision offers of reducing memory requirements.
Alternatively, the entries can be cast on the fly. This
is straightforward to incorporate into a serial triangular solve routine, and 
only requires a temporary double precision array of length $n$. This is done in our software.

We use two precisions: $u_{\ell}=u_{16}$  for the incomplete factorization and {$u_r = u_g = u_p = u = u_{64}$, where $u_{16}$ and $u_{64}$ denote the unit roundoffs in fp16 and fp64 arithmetic, respectively.
That is, we aim to achieve double precision accuracy in the computed solution.
The iterative refinement terminates when the
normwise backward error for the computed solution satisfies
\begin{equation*}
res = \frac{\| b - A  x\|_\infty}{\| A\|_\infty \|  x\|_\infty + \| b\|_\infty} \le \delta = 10^3 \times u_{64}.
 \end{equation*}
The implementation of GMRES is taken from the HSL software library \cite{hsl:2023}
({\tt MI24} is a Fortran MGS-GMRES implementation),
and its convergence tolerance is set to $u_{64}^{1/4}$ (see \cite{sctu:2024a} for
an explanation of this choice); for each application of GMRES (Step 5 of Algorithm~\ref{alg:gmres-ir-ic})
the limit on the number of iterations is 1000. Restarting is not used.
In the tables of results, NC denotes that this limit has been exceeded without the GMRES convergence tolerance being achieved.

Our first experiment looks at the effects of incorporating look-ahead.
Table~\ref{T:look-ahead}  presents results  with no look-ahead and with look-ahead. Here we only include
the test problems for which look-ahead has an effect. We make a number of
observations. Incorporating look-ahead can improve robustness,
particularly when using fp64 arithmetic.
Without look-ahead, in fp64 arithmetic there can be very large growth
in the size of some entries in the factors and this goes undetected (no B1 breakdowns occur with $\tau_u = 10^{-20}$).
With look-ahead, growth did not happen in our tests on ill-conditioned problems.
In fp16 arithmetic, look-ahead can replace B3 breakdown by B1 breakdown (e.g., HB/bcsstk24).
Even if there are no B3 breakdowns, look-ahead can lead to a larger number of 
B1 breakdowns being flagged (see problem Boeing/msc01050) and hence a larger number
of restarts, a larger shift and, consequently, a higher GMRES iteration count.
Note that the iteration counts for $IC(3)$ are not guaranteed to be smaller than for $IC(2)$
(although they generally are).

\begin{table}[htbp]
\caption{Results for IC-GMRES-IR 
(Algorithm~\ref{alg:gmres-ir-ic}) using fp16-$IC(\ell)$ 
and fp64-$IC(\ell)$ preconditioners 
($\ell = 2$, 3) with no look-ahead and with look-ahead (Section~\ref{subsec:look_ahead}).
$its$ is the total number of GMRES iterations and 
$(n1,n2)$ are the numbers of times B1 and B3 breakdowns are detected ($n2$ is nonzero only for fp16-$IC(\ell)$ with no look-ahead).
NC indicates that on an inner iteration the requested GMRES accuracy of $u_{64}^{1/4}$ was not achieved within the 
limit of 1000 iterations.
\ddag~indicates failure to compute 
useful factors because of enormous growth in the entries.
}
\label{T:look-ahead}\vspace{3mm}
{
\footnotesize
\begin{tabular}{lrr|rr} 
  \hline

  \multicolumn{1}{l}{Identifier}  &
    \multicolumn{1}{c}{No look-ahead}  &
  \multicolumn{1}{c|}{With look-ahead}  &
  \multicolumn{1}{c}{No look-ahead}  &
  \multicolumn{1}{c}{With look-ahead}  \\
 &
\multicolumn{1}{c}{$its$ $(n1,n2)$}  &
\multicolumn{1}{r|}{$its$ $(n1)$}  &
\multicolumn{1}{c}{$its$ $(n1,n2)$}  &
\multicolumn{1}{r}{$its$ $(n1)$}   \\
\hline\Tstrut
    &
  \multicolumn{2}{c|}{fp16-$IC(2)$}  &
  \multicolumn{2}{c}{fp16-$IC(3)$}  \\
\hline\Tstrut
Boeing/msc01050  &  65 (1, 0)   & 84  (4) & 62 (1, 0)   & 81 (4)   \\ 
HB/bcsstk24      & 428 (0, 1)   & 428 (1) & 418 (0, 1)  & 418 (1)  \\  
Um/offshore      &  2013 (0, 4)   & 129 (5)  &  40 (0, 4) &  40 (4) \\ 
\hline\Tstrut
&
  \multicolumn{2}{c|}{fp64-$IC(2)$}  &
  \multicolumn{2}{c}{fp64-$IC(3)$}  \\
\hline\Tstrut
Boeing/msc01050  & 24 (0, 0)  & 69 (4) & 25 (0,0)  & 69 (4)  \\ 
HB/bcsstk11      & 201 (0, 0) & 174 (1) & 29 (0,0)  & 29 (0) \\ 
Cylshell/s3rmt3m1   & 102 (0, 0) & 102 (0)  & NC (0, 0) & 426 (1)  \\ 
Boeing/ct20stif    & NC (0, 0) & 1940 (2) & 1332 (0, 0) & 1368 (1) \\
GHS$\_$psdef/hood    & \ddag~ (0, 0) & 568 (5) & \ddag~ (0, 0) & 407 (4) \\ 
Um/offshore          & \ddag~ (0, 0) & 128 (5) & \ddag~ (0, 0) &  37 (4)\\ 
\hline
\end{tabular}
}
\end{table}

Table~\ref{T:gmres_fp16new_ic2} presents results for  IC-GMRES-IR using a fp16-$IC(2)$ preconditioner with look-ahead and the
GMW($\beta$) strategy for $\beta = 0.5$, 10, and 100.
B3 breakdown only occurs for GMW(100) (there is a single B3 breakdown
for examples HB/bcsstk24 and Boeing/bcsstk38 and  for these a global shift is used).
B4 breakdown happens only for GMW(10) applied to  GHS\_psdef/hood (5 occurrences  for this example)
and GMW(100) applied to HB/bcsstk11 (happens once); again a global shift is used to avoid breakdown.
We see that with $\beta = 0.5$, for some examples a large number  of local
modifications ($nmod$) are made. This leads to the preconditioner being
of poorer quality compared to the $IC(2)$ preconditioner with look-ahead.
For $\beta = 10$, local modifications are only needed for a few problems 
(HB/bcsstk11, HB/bcsstk24 and Um/offshore). In each case, the resulting preconditioner is not successful. For GMW(100), local modifications are only made for Um/offshore; for
all other test examples, GMW(100) is equivalent to $IC(2)$ with no look-ahead.

\begin{table}[htbp]
\caption{Results for IC-GMRES-IR (Algorithm~\ref{alg:gmres-ir-ic})
using a fp16-$IC(2)$ preconditioner with the
GMW($\beta$) strategy (Section~\ref{subsec:modified}) and  with look-ahead (Section~\ref{subsec:look_ahead}).
$its$ is the number of GMRES iterations and 
$(n1,nmod)$ are the numbers of times B1 breakdown is detected
and the number of local modifications made by the GMW strategy. 
$^\sharp$ 
and $^\ast$  indicate B3 and B4 breakdowns, respectively.
NC indicates that on an inner iteration the requested GMRES accuracy of $u_{64}^{1/4}$
was not achieved within the 
limit of 1000 iterations.
}
\label{T:gmres_fp16new_ic2}\vspace{3mm}
{
\footnotesize
\begin{tabular}{lrrrr} 
  \hline\Tstrut
{Identifier} &
{GMW(0.5)} & 
{GMW(10)} & 
{GMW(100)} & 
With look-ahead  \\
\multicolumn{1}{c}{}  &
\multicolumn{1}{r}{$its$ $(n1, nmod)$}  &
\multicolumn{1}{r}{$its$ $(n1, nmod)$}  &
\multicolumn{1}{r}{$its$ $(n1, nmod)$}  &
\multicolumn{1}{r}{$its$ $(n1)$}    \\
\hline\Tstrut
Boeing/msc01050     & 96 (0, 60)    &  65 (1, 0) &  65 (1, 0) & 84 (4)   \\ 
HB/bcsstk11         & 1092 (0, 476) & NC (0, 310) &  205$^\ast$ (0, 0) &  205 (1) \\ 
HB/bcsstk26         & 786 (0, 476)     & 111 (1, 0) & 111 (1, 0)  & 87 (1)    \\ 
HB/bcsstk24         & 1018 (0, 446)  & NC (0, 428) & 428$^\sharp$ (0, 0) & 428 (1) \\ 
HB/bcsstk16            & 41 (0, 26)     & 23 (0, 0) & 23 (0, 0)  & 23 (0)     \\
Cylshell/s2rmt3m1      & 787 (0, 584)   & 155 (0, 0) & 155 (0, 0) & 155 (0)    \\ 
Cylshell/s3rmt3m1      & 2017  (1, 710)  & 630  (2, 0)   & 630 (2, 0) & 630 (2)   \\ 
Boeing/bcsstk38        & 1335 (1, 914)  & 313 (1, 0)  & 313$^\sharp$ (0, 0) & 313 (1)    \\
Boeing/msc10848        & 684 (0, 591)   & 81 (0, 0)  & 81 (0, 0)  & 81 (0)   \\ 
Oberwolfach/t2dah$\_$e & 11 (0, 6)    &  7 (0, 0)    & 7 (0, 0)   & 7 (0)   \\
Boeing/ct20stif        & 2139 (0, 4827)   & 1900 (2, 0)   & 1900 (2, 0)  & 1900 (2) \\ 
DNVS/shipsec8          & 2569 (1, 1)  & 2390 (1, 0)   & 2390 (1, 0) & 1492 (1)    \\ 
GHS$\_$psdef/hood      & 2459 (0, 25074) & 581$^\ast$ (0, 0) & 581 (5, 0)  & 581 (5)   \\ 
Um/offshore            &  NC (0, 3846)    & NC (0, 5838)  &  2013 (4, 2) &  129 (5)    \\

\hline
\end{tabular}
}
\end{table}

The sensitivity of the GMW($\beta$) approach to the
choice of $\beta$ is reported on in Table~\ref{T:gmres_fp16_gmw} for problem Boeing/bcsstk38. As $\beta$ increases, the number of local
modifications to diagonal entries ($nmod$) steadily decreases 
and so too does the GMRES iteration count ($its$). For this example,
for each $\beta \ge 0.4$, B1 breakdown was detected once and a
global shift $\alpha = 10^{-3}$ was then employed. 

\begin{table}[htbp]
\caption{Results for problem Boeing/bcsstk38.  IC-GMRES-IR is run using a fp16-$IC(2)$ preconditioner computed with the
GMW($\beta$) strategy for a range of values of $\beta$ (Section~\ref{subsec:modified}).
$nmod$ and $its$ are the numbers of local modifications made by the GMW strategy and GMRES iterations, respectively.
}
\label{T:gmres_fp16_gmw}\vspace{3mm}
{
\footnotesize
\begin{tabular}{l|rrrr rrrr rrr} 
\hline\Tstrut
$\beta$   &  0.1 & 0.2  & 0.3 & 0.4  & 0.45 & 0.5 & 0.6 & 0.7 & 0.8 & 0.9 & 1.0 \\
\hline\Tstrut

$nmod$ &7404 & 5751 & 3695 & 1891 & 1327 & 914 & 621 & 409 & 138 & 8 & 0 \\
$its$ & 3542 & 2634 & 2068 & 2037 & 1642 & 1335 & 481 & 390 & 321 & 313 & 313 \\
\hline
\end{tabular}
}
\end{table}

Finally, results for IC-GMRES-IR using a fp64-$IC(2)$ preconditioner
are given in Table~\ref{T:gmres_fp64new_ic2}. As we would expect,
the number of breakdowns and the iteration counts are
often less than for the fp16-$IC(2)$ preconditioner.
If $\beta = 0.5$, the number of local modifications when using fp64
arithmetic is very similar to the number when using fp16 and the iteration
counts are also comparable. For larger $\beta$, fp64 can result
in a higher quality preconditioner but, as earlier, without look-ahead
the computed preconditioner can be ineffective. 

\begin{table}[htbp]
\caption{Results for IC-GMRES-IR (Algorithm~\ref{alg:gmres-ir-ic})
using a fp64-$IC(2)$ preconditioner with the
GMW($\beta$) strategy (Section~\ref{subsec:modified}) and with look-ahead  (Section~\ref{subsec:look_ahead}).
$its$ is the number of GMRES iterations and 
$(n1,nmod)$ are the numbers of times B1 breakdown is detected
and the number of local modifications made by the GMW strategy. 
For GMW(0.5) and GMW(10), $n1= 0$ for all examples so is omitted.
NC indicates that on an inner iteration the requested GMRES accuracy of $u_{64}^{1/4}$
was  not achieved within the limit of 1000 iterations.
\ddag~indicates failure to compute 
useful factors because of enormous growth in the entries.}
\label{T:gmres_fp64new_ic2}\vspace{3mm}
{
\footnotesize
\begin{tabular}{lrrrr} 
  \hline\Tstrut
{Identifier} &
{GMW(0.5)} & 
{GMW(10)} & 
{GMW(100)} & 
{With look-ahead}  \\
\multicolumn{1}{c}{}  &
\multicolumn{1}{c}{$its$ $(nmod)$}  &
\multicolumn{1}{c}{$its$ $(nmod)$}  &
\multicolumn{1}{c}{$its$ $(n1, nmod)$}  &
\multicolumn{1}{c}{$its$ $(n1)$}   \\
\hline\Tstrut
Boeing/msc01050       & 78 (60) &  24 (0)     & 24 (4, 0) & 24 (0)  \\ 
HB/bcsstk11            & 1087 (476) &  201 (0)  & 201 (0, 0) & 232 (0)  \\ 
HB/bcsstk26            & 775 (476) &  79 (0)     &79 (0, 0)  & 79 (0)  \\ 
HB/bcsstk24           & 913 (409) &  89 (0)     &89 (0, 0) & 89 (0)   \\ 
HB/bcsstk16           &  41 (26) &  22 (0)       &22 (0, 0) & 22 (0)  \\
Cylshell/s2rmt3m1     & 792 (585) &  146 (0)    & 146 (0, 0) & 146 (0)  \\ 
Cylshell/s3rmt3m1     & 2901 (710) & 102 (0)    & 102 (0, 0) & 102 (0)  \\ 
Boeing/bcsstk38       & 1301 (943) & 141 (0)   & 141 (0, 0) & 141 (0)  \\ 
Boeing/msc10848       & 790 (600) & 68 (0)      &  68 (0, 0) & 68 (0) \\ 
Oberwolfach/t2dah$\_$e  & 14 (6) &  6 (0)         & 6 (0, 0) & 6 (0) \\ 
Boeing/ct20stif       & 2122 (4847) & 2036 (40) & NC (0, 0) & 1940 (2)  \\ 
DNVS/shipsec8         & 701 (4658) & 354 (0)        & 354 (0, 55) & 354 (0) \\ 
GHS$\_$psdef/hood     & 2480 (25054) &  NC (2013) &   NC  (0, 11998) & 568 (5)  \\ 
Um/offshore           & NC (4094) & NC (6327) &  \ddag (4, 4730) & 128 (5)  \\ 
\hline
\end{tabular}
}
\end{table}

All the reported results employed our explicit safe tests for breakdown.
We have also run the fp16 arithmetic experiments with the B1 to B4 breakdown tests replaced
by IEEE exception handling. As expected, because in fp16 arithmetic
$\tau_u = 10^{-5}$ and $x_{min} = \mathcal{O}(10^{-5})$, this led to the same
number of restarts and hence the same iteration counts. However, 
by only testing the exception flag at the end of each major step
of the factorization,
this approach
did not distinguish between the different types of breakdown.
For the experiments using fp64 arithmetic, IEEE exception handling
did not detect any problems and consequently, for some examples,
this led to growth in the factor entries (exactly as in the case
of no look-ahead).

\section{Concluding remarks}
\label{sec:conclusions}
Following on from our earlier study \cite{sctu:2024a}, in this
paper we have illustrated the potential for 
using half precision arithmetic to compute incomplete factorization
preconditioners that can be used to obtain double precision
accuracy in the solution of highly ill-conditioned
symmetric positive definite linear systems.
In fp16 arithmetic, the danger of breakdown during the
factorization of a sparse matrix is imminent and
we must employ strategies that force computational robustness. 
To avoid breakdown, we have looked at global strategies
plus a local modification scheme
based on the GMW approach that has been  used for dense matrices
within the field of optimization. 
This employs a parameter $\beta$.
Choosing a small $\beta$ prevents breakdown during the factorization
(in both fp16 and fp64 arithmetic) and there is no 
need to employ a global shift. However, the penalty  is
of poorer quality than that which is obtained by employing a simple global shifting
approach. 
Thus, our recommendations are to always prescale the problem, to use a global shift, and to incorporate look-ahead. In addition, when developing software using fp16,  monitoring  for breakdown must be built in to ensure robustness. If this is done, then using low precision to compute
an effective preconditioner appears to be feasible.

Once a Fortran compiler that supports bfloat16 becomes available, it would be very interesting to compare its performance to that of fp16.
 bfloat16 has the same exponent size as fp32 (single precision). Consequently, converting from fp32 to bfloat16 is easy: the exponent is kept the same and the significand is rounded or truncated from 24 bits to 8; hence overflow and underflow are not possible in the conversion.
The disadvantage of bfloat16 is its lesser precision: essentially 3 significant decimal digits versus 4 for fp16.
Another possible future direction is to explore the effects of different preorderings of $A$ on the number of breakdowns and the quality of the low precision factors. Fill-reducing orderings can result in later entries 
in the factor being updated by more entries from the previous columns. Intuitively, this may lead to more breakdowns.

When using higher precision arithmetic, the potential dangers within an
incomplete factorization algorithm can be hidden. 
As our experiments have demonstrated, a standard $IC$ factorization 
using fp64 arithmetic without look-ahead  can lead to an ineffective
preconditioner because of
growth in the size of the entries in the factors. Without careful monitoring
(which is not routinely done), this growth may be unobserved
but when subsequently applying the preconditioner, the triangular solves can overflow, resulting in the computation aborting. 

Finally, we reiterate that, although our focus has been on symmetric positive definite systems,
breakdown and/or large growth in factor entries is also
an issue for the incomplete factorization of general sparse matrices.
Again, safe checks (or the use of IEEE exception handling) need to be built into the algorithms and their
implementations to guarantee robustness.

\bmhead{Acknowledgements}


We would like to thank John Reid for discussions on exception handling 
in Fortran and
for carefully reading and commenting on a draft of this paper.
We are also grateful to two anonymous referees for their insightful feedback that has led to improvements in the paper.

\section*{Declarations}


\begin{itemize}
\item Funding. The authors received no external grant funding for this research.
\item Ethics approval and consent to participate. There was no ethics approval required for this research.
\item Data. The test matrices used in this study are taken from the  SuiteSparse Collection
and are available at \url{https://sparse.tamu.edu/}
\end{itemize}





\small

\bibliographystyle{sn-bibliography}


\begin{thebibliography}{33}
\def\cprime{$'$} \def\cprime{$'$} \def\cprime{$'$}
\ifx \bisbn   \undefined \def \bisbn  #1{ISBN #1}\fi
\ifx \binits  \undefined \def \binits#1{#1}\fi
\ifx \bauthor  \undefined \def \bauthor#1{#1}\fi
\ifx \batitle  \undefined \def \batitle#1{#1}\fi
\ifx \bjtitle  \undefined \def \bjtitle#1{#1}\fi
\ifx \bvolume  \undefined \def \bvolume#1{\textbf{#1}}\fi
\ifx \byear  \undefined \def \byear#1{#1}\fi
\ifx \bissue  \undefined \def \bissue#1{#1}\fi
\ifx \bfpage  \undefined \def \bfpage#1{#1}\fi
\ifx \blpage  \undefined \def \blpage #1{#1}\fi
\ifx \burl  \undefined \def \burl#1{\textsf{#1}}\fi
\ifx \doiurl  \undefined \def \doiurl#1{\url{https://doi.org/#1}}\fi
\ifx \betal  \undefined \def \betal{\textit{et al.}}\fi
\ifx \binstitute  \undefined \def \binstitute#1{#1}\fi
\ifx \binstitutionaled  \undefined \def \binstitutionaled#1{#1}\fi
\ifx \bctitle  \undefined \def \bctitle#1{#1}\fi
\ifx \beditor  \undefined \def \beditor#1{#1}\fi
\ifx \bpublisher  \undefined \def \bpublisher#1{#1}\fi
\ifx \bbtitle  \undefined \def \bbtitle#1{#1}\fi
\ifx \bedition  \undefined \def \bedition#1{#1}\fi
\ifx \bseriesno  \undefined \def \bseriesno#1{#1}\fi
\ifx \blocation  \undefined \def \blocation#1{#1}\fi
\ifx \bsertitle  \undefined \def \bsertitle#1{#1}\fi
\ifx \bsnm \undefined \def \bsnm#1{#1}\fi
\ifx \bsuffix \undefined \def \bsuffix#1{#1}\fi
\ifx \bparticle \undefined \def \bparticle#1{#1}\fi
\ifx \barticle \undefined \def \barticle#1{#1}\fi
\bibcommenthead
\ifx \bconfdate \undefined \def \bconfdate #1{#1}\fi
\ifx \botherref \undefined \def \botherref #1{#1}\fi
\ifx \bchapter \undefined \def \bchapter#1{#1}\fi
\ifx \bbook \undefined \def \bbook#1{#1}\fi
\ifx \bcomment \undefined \def \bcomment#1{#1}\fi
\ifx \oauthor \undefined \def \oauthor#1{#1}\fi
\ifx \citeauthoryear \undefined \def \citeauthoryear#1{#1}\fi
\ifx \endbibitem  \undefined \def \endbibitem {}\fi
\ifx \bconflocation  \undefined \def \bconflocation#1{#1}\fi
\ifx \arxivurl  \undefined \def \arxivurl#1{\textsf{#1}}\fi
\csname PreBibitemsHook\endcsname

\bibitem[\protect\citeauthoryear{Higham and Mary}{2022}]{hima:2022}
\begin{barticle}
\bauthor{\bsnm{Higham}, \binits{N.J.}},
\bauthor{\bsnm{Mary}, \binits{T.}}:
\batitle{Mixed precision algorithms in numerical linear algebra}.
\bjtitle{Acta Numerica}
\bvolume{31},
\bfpage{347}--\blpage{414}
(\byear{2022})
\doiurl{10.1017/S0962492922000022}
\end{barticle}
\endbibitem

\bibitem[\protect\citeauthoryear{Carson and Higham}{2017}]{cahi:2017}
\begin{barticle}
\bauthor{\bsnm{Carson}, \binits{E.}},
\bauthor{\bsnm{Higham}, \binits{N.J.}}:
\batitle{A new analysis of iterative refinement and its application to accurate
  solution of ill-conditioned sparse linear systems}.
\bjtitle{SIAM J. on Scientific Computing}
\bvolume{39}(\bissue{6}),
\bfpage{2834}--\blpage{2856}
(\byear{2017})
\doiurl{10.1137/17M1122918}
\end{barticle}
\endbibitem

\bibitem[\protect\citeauthoryear{Amestoy et~al.}{2024}]{ames:2024}
\begin{barticle}
\bauthor{\bsnm{Amestoy}, \binits{P.}},
\bauthor{\bsnm{Buttari}, \binits{A.}},
\bauthor{\bsnm{Higham}, \binits{N.J.}},
\bauthor{\bsnm{L’Excellent}, \binits{J.-Y.}},
\bauthor{\bsnm{Mary}, \binits{T.}},
\bauthor{\bsnm{Vieubl{\'e}}, \binits{B.}}:
\batitle{Five-precision {GMRES}-based iterative refinement}.
\bjtitle{SIAM J. on Matrix Analysis and Applications}
\bvolume{45}(\bissue{1}),
\bfpage{529}--\blpage{552}
(\byear{2024})
\doiurl{10.1137/23M1549079}
\end{barticle}
\endbibitem

\bibitem[\protect\citeauthoryear{Carson and Higham}{2018}]{cahi:2018}
\begin{barticle}
\bauthor{\bsnm{Carson}, \binits{E.}},
\bauthor{\bsnm{Higham}, \binits{N.J.}}:
\batitle{Accelerating the solution of linear systems by iterative refinement in
  three precisions}.
\bjtitle{SIAM J. on Scientific Computing}
\bvolume{40}(\bissue{2}),
\bfpage{817}--\blpage{847}
(\byear{2018})
\doiurl{10.1137/17M1140819}
\end{barticle}
\endbibitem

\bibitem[\protect\citeauthoryear{Higham and Pranesh}{2021}]{hipr:2021}
\begin{barticle}
\bauthor{\bsnm{Higham}, \binits{N.J.}},
\bauthor{\bsnm{Pranesh}, \binits{S.}}:
\batitle{Exploiting lower precision arithmetic in solving symmetric positive
  definite linear systems and least squares problems}.
\bjtitle{SIAM J. on Scientific Computing}
\bvolume{43}(\bissue{1}),
\bfpage{258}--\blpage{277}
(\byear{2021})
\doiurl{10.1137/19M1298263}
\end{barticle}
\endbibitem

\bibitem[\protect\citeauthoryear{Scott and T\r{u}ma}{2024}]{sctu:2024a}
\begin{botherref}
\oauthor{\bsnm{Scott}, \binits{J.A.}},
\oauthor{\bsnm{T\r{u}ma}, \binits{M.}}:
Avoiding breakdown in incomplete factorizations in low precision arithmetic.
ACM Transactions on Mathematical Software, {\bf 50}(2), Article 9
(2024). 
\doiurl{10.1145/3651155}
\end{botherref}
\endbibitem

\bibitem[\protect\citeauthoryear{Eijkhout}{1999}]{eijk:99}
\begin{botherref}
\oauthor{\bsnm{Eijkhout}, \binits{V.}}:
On the existence problem of incomplete factorisation methods.
{Lapack Working Note No. 144, UT-CS-99-435},
  \url{http://www.netlib.org/lapack/lawns/}
(1999)
\end{botherref}
\endbibitem

\bibitem[\protect\citeauthoryear{Higham and Pranesh}{2019}]{hipr:2019}
\begin{barticle}
\bauthor{\bsnm{Higham}, \binits{N.J.}},
\bauthor{\bsnm{Pranesh}, \binits{S.}}:
\batitle{Simulating low precision floating-point arithmetic}.
\bjtitle{SIAM J. on Scientific Computing}
\bvolume{41}(\bissue{5}),
\bfpage{585}--\blpage{602}
(\byear{2019})
\doiurl{10.1137/19M1251308}
\end{barticle}
\endbibitem

\bibitem[\protect\citeauthoryear{Chan and van~der Vorst}{1997}]{chvd:97}
\begin{bchapter}
\bauthor{\bsnm{Chan}, \binits{T.F.}},
\bauthor{\bsnm{Vorst}, \binits{H.A.}}:
\bctitle{Approximate and incomplete factorizations}.
In: \bbtitle{Parallel Numerical Algorithms, ICASE/LaRC Interdisciplinary Series
  in Science and Engineering IV. Centenary Conference, D.E. Keyes, A. Sameh and
  V. Venkatakrishnan, Eds.},
pp. \bfpage{167}--\blpage{202}.
\bpublisher{Kluver Academic Publishers},
\blocation{Dordrecht}
(\byear{1997}).
\doiurl{10.1007/978-94-011-5412-3\_6}
\end{bchapter}
\endbibitem

\bibitem[\protect\citeauthoryear{Saad}{2003}]{saad:03}
\begin{bbook}
\bauthor{\bsnm{Saad}, \binits{Y.}}:
\bbtitle{Iterative Methods for Sparse Linear Systems},
\bedition{2}nd edn.,
p. \bfpage{528}.
\bpublisher{SIAM},
\blocation{Philadelphia, PA}
(\byear{2003}).
\doiurl{10.1137/1.9780898718003}
\end{bbook}
\endbibitem

\bibitem[\protect\citeauthoryear{Il{\cprime}in}{1992}]{illi:92}
\begin{bbook}
\bauthor{\bsnm{Il{\cprime}in}, \binits{Y.M.}}:
\bbtitle{Iterative Incomplete Factorization Methods}.
\bpublisher{World Scientific},
\blocation{Singapore}
(\byear{1992}).
\doiurl{10.1142/1677}
\end{bbook}
\endbibitem

\bibitem[\protect\citeauthoryear{Scott and T\accent23uma}{2011}]{sctu:2011}
\begin{barticle}
\bauthor{\bsnm{Scott}, \binits{J.A.}},
\bauthor{\bsnm{T\accent23uma}, \binits{M.}}:
\batitle{The importance of structure in incomplete factorization
  preconditioners}.
\bjtitle{BIT Numerical Mathematics}
\bvolume{51},
\bfpage{385}--\blpage{404}
(\byear{2011})
\doiurl{10.1007/s10543-010-0299-8}
\end{barticle}
\endbibitem

\bibitem[\protect\citeauthoryear{Scott and T\r{u}ma}{}]{sctu:2023a}
\begin{botherref}
\oauthor{\bsnm{Scott}, \binits{J.A.}},
\oauthor{\bsnm{T\r{u}ma}, \binits{M.}}:
Algorithms for Sparse Linear Systems.
Ne\v{c}as Center Series,
p. 242.
Birkh\"{a}user/Springer, Cham, 2023.
\doiurl{10.1007/978-3-031-25820-6}
\end{botherref}
\endbibitem

\bibitem[\protect\citeauthoryear{Meijerink and van der Vorst}{1977}]{meva:77}
\begin{barticle}
\bauthor{\bsnm{Meijerink}, \binits{J.A.}},
\bauthor{\bsnm{van der Vorst}, \binits{H.A.}}:
\batitle{An iterative solution method for linear systems of which the coefficient matrix is a symmetric {$M$}-matrix}.
\bjtitle{Mathematics of Computation}
\bvolume{31},
\bfpage{148}--\blpage{162}
(\byear{1977})
\end{barticle}

\bibitem[\protect\citeauthoryear{Meurant}{1999}]{meur:99}
\begin{bbook}
\bauthor{\bsnm{Meurant}, \binits{G.}}:
\bbtitle{Computer Solution of Large Linear Systems}.
\bpublisher{Elsevier},
\blocation{Amsterdam -- Lausanne -- New York -- Oxford -- Shannon -- Singapore -- Tokyo}
(\byear{1999}).
\end{bbook}
\endbibitem


\bibitem[\protect\citeauthoryear{Oktay and Carson}{2022}]{okca:2022}
\begin{barticle}
\bauthor{\bsnm{Oktay}, \binits{E.}},
\bauthor{\bsnm{Carson}, \binits{E.}}:
\batitle{Multistage mixed precision iterative refinement}.
\bjtitle{Numerical Linear Algebra with Applications}
\bvolume{29}(\bissue{4}),
\bfpage{2434}
(\byear{2022})
\doiurl{10.1002/nla.2434}
\end{barticle}
\endbibitem

\bibitem[\protect\citeauthoryear{Demmel and Li}{1994}]{deli:94}
\begin{barticle}
\bauthor{\bsnm{Demmel}, \binits{J.W.}},
\bauthor{\bsnm{Li}, \binits{X.}}:
\batitle{Faster numerical algorithms via exception handling}.
\bjtitle{IEEE Transactions on Computers}
\bvolume{43}(\bissue{8}),
\bfpage{983}--\blpage{992}
(\byear{1994})
\doiurl{10.1109/12.295860}
\end{barticle}
\endbibitem

\bibitem[\protect\citeauthoryear{Lindstrom et~al.}{2018}]{lilh:2018}
\begin{bchapter}
\bauthor{\bsnm{Lindstrom}, \binits{P.}},
\bauthor{\bsnm{Lloyd}, \binits{S.}},
\bauthor{\bsnm{Hittinger}, \binits{J.}}:
\bctitle{Universal coding of the reals: alternatives to {IEEE} floating point}.
In: \bbtitle{Proceedings of the Conference for Next Generation Arithmetic},
pp. \bfpage{1}--\blpage{14}
(\byear{2018}).
\doiurl{10.1145/3190339.3190344}
\end{bchapter}
\endbibitem

\bibitem[\protect\citeauthoryear{Higham et~al.}{2019}]{hipz:2019}
\begin{barticle}
\bauthor{\bsnm{Higham}, \binits{N.J.}},
\bauthor{\bsnm{Pranesh}, \binits{S.}},
\bauthor{\bsnm{Zounon}, \binits{M.}}:
\batitle{Squeezing a matrix into half precision, with an application to solving
  linear systems}.
\bjtitle{SIAM J. on Scientific Computing}
\bvolume{41}(\bissue{4}),
\bfpage{2536}--\blpage{2551}
(\byear{2019})
\doiurl{10.1137/18M1229511}
\end{barticle}
\endbibitem

\bibitem[\protect\citeauthoryear{Lin and Mor{\'e}}{1999}]{limo:99}
\begin{barticle}
\bauthor{\bsnm{Lin}, \binits{C.-J.}},
\bauthor{\bsnm{Mor{\'e}}, \binits{J.J.}}:
\batitle{Incomplete {C}holesky factorizations with limited memory}.
\bjtitle{SIAM J. on Scientific Computing}
\bvolume{21}(\bissue{1}),
\bfpage{24}--\blpage{45}
(\byear{1999})
\doiurl{10.1137/S1064827597327334}
\end{barticle}
\endbibitem

\bibitem[\protect\citeauthoryear{Scott and T\accent23uma}{2014}]{sctu:2014a}
\begin{barticle}
\bauthor{\bsnm{Scott}, \binits{J.A.}},
\bauthor{\bsnm{T\accent23uma}, \binits{M.}}:
\batitle{{\tt HSL\_MI28}: an efficient and robust limited-memory incomplete
  {C}holesky factorization code}.
\bjtitle{ACM Transactions on Mathematical Software}
{\bf 40}(4), Article 24
(2014).
\doiurl{10.1145/2617555}
\end{barticle}
\endbibitem

\bibitem[\protect\citeauthoryear{Manteuffel}{1980}]{mant:80}
\begin{barticle}
\bauthor{\bsnm{Manteuffel}, \binits{T.A.}}:
\batitle{An incomplete factorization technique for positive definite linear
  systems}.
\bjtitle{Mathematics of Computation}
\bvolume{34},
\bfpage{473}--\blpage{497}
(\byear{1980})
\doiurl{10.2307/2006097}
\end{barticle}
\endbibitem

\bibitem[\protect\citeauthoryear{Gill and Murray}{1974}]{gimu:74}
\begin{barticle}
\bauthor{\bsnm{Gill}, \binits{P.E.}},
\bauthor{\bsnm{Murray}, \binits{W.}}:
\batitle{Newton-type methods for unconstrained and linearly constrained
  optimization}.
\bjtitle{Mathematical Programming}
\bvolume{7},
\bfpage{311}--\blpage{350}
(\byear{1974})
\doiurl{10.1007/BF01585529}
\end{barticle}
\endbibitem

\bibitem[\protect\citeauthoryear{Gill et~al.}{2019}]{gimw:81}
\begin{bbook}
\bauthor{\bsnm{Gill}, \binits{P.E.}},
\bauthor{\bsnm{Murray}, \binits{W.}},
\bauthor{\bsnm{Wright}, \binits{M.H.}}:
\bbtitle{Practical Optimization}.
\bpublisher{SIAM, 2nd edition},
\blocation{Philadelphia}
(\byear{2019}).
\doiurl{10.1137/1.9781611975604}
\end{bbook}
\endbibitem

\bibitem[\protect\citeauthoryear{Eskow and Schnabel}{1991}]{essc:91}
\begin{barticle}
\bauthor{\bsnm{Eskow}, \binits{E.}},
\bauthor{\bsnm{Schnabel}, \binits{R.B.}}:
\batitle{Algorithm 695: Software for a new modified {C}holesky factorization}.
\bjtitle{ACM Transactions on Mathematical Software}
\bvolume{17},
\bfpage{306}--\blpage{312}
(\byear{1991})
\doiurl{10.1145/114697.116806}
\end{barticle}
\endbibitem


\bibitem[\protect\citeauthoryear{Conn et~al.}{1996}]{cogt:96}
\begin{barticle}
\bauthor{\bsnm{Conn}, \binits{A.R.}},
\bauthor{\bsnm{Gould}, \binits{N.I.M.}},
\bauthor{\bsnm{Toint}, \binits{Ph.L.}}:
\batitle{Numerical experiments with the {LANCELOT} package ({Release A}) for large-scale nonlinear optimization}.
\bjtitle{Mathematical Programming}
\bvolume{73},
\bfpage{73}--\blpage{110}
(\byear{1996})
\doiurl{10.1007/BF02592099}

\end{barticle}
\endbibitem


\bibitem[\protect\citeauthoryear{Fang and O'Leary}{2008}]{faol:08}
\begin{barticle}
\bauthor{\bsnm{Fang}, \binits{H.}},
\bauthor{\bsnm{O'Leary}, \binits{D.P.}}:
\batitle{Modified {C}holesky algorithms: a catalog with new approaches}.
\bjtitle{Mathematical Programming}
\bvolume{115}(\bissue{2, Ser. A}),
\bfpage{319}--\blpage{349}
(\byear{2008})
\doiurl{10.1007/s10107-007-0177-6}
\end{barticle}
\endbibitem

\bibitem[\protect\citeauthoryear{Schnabel and Eskow}{1990a}]{sces:90}
\begin{barticle}
\bauthor{\bsnm{Schnabel}, \binits{R.B.}},
\bauthor{\bsnm{Eskow}, \binits{E.}}:
\batitle{A new modified {C}holesky factorization}.
\bjtitle{SIAM J. on Scientific Computing}
\bvolume{11},
\bfpage{424}--\blpage{445}
(\byear{1990})
\doiurl{10.1137/0911064}
\end{barticle}
\endbibitem

\bibitem[\protect\citeauthoryear{Schnabel and Eskow}{1990b}]{sces:99}
\begin{barticle}
\bauthor{\bsnm{Schnabel}, \binits{R.B.}},
\bauthor{\bsnm{Eskow}, \binits{E.}}:
\batitle{A revised modified {C}holesky factorization algorithm}.
\bjtitle{SIAM J. on Optimization}
\bvolume{11},
\bfpage{1135}--\blpage{1148}
(\byear{1990})
\doiurl{10.1137/S105262349833266X}
\end{barticle}
\endbibitem

\bibitem[\protect\citeauthoryear{Greenbaum}{1997}]{gree:97a}
\begin{barticle}
\bauthor{\bsnm{Greenbaum}, \binits{A.}}:
\batitle{Estimating the attainable accuracy of recursively computed residual
  methods}.
\bjtitle{SIAM J. on Matrix Analysis and Applications}
\bvolume{18},
\bfpage{535}--\blpage{551}
(\byear{1997})
\doiurl{10.1137/S0895479895284944}
\end{barticle}
\endbibitem

\bibitem[\protect\citeauthoryear{Carson et~al.}{2020}]{cahp:2020a}
\begin{barticle}
\bauthor{\bsnm{Carson}, \binits{E.}},
\bauthor{\bsnm{Higham}, \binits{N.J.}},
\bauthor{\bsnm{Pranesh}, \binits{S.}}:
\batitle{Three-precision {GMRES}-based iterative refinement for least squares
  problems}.
\bjtitle{SIAM J. on Scientific Computing}
\bvolume{42}(\bissue{6}),
\bfpage{4063}--\blpage{4083}
(\byear{2020})
\doiurl{10.1137/20M1316822}
\end{barticle}
\endbibitem

\bibitem[\protect\citeauthoryear{Carson and Khan}{2023}]{cakh:2023}
\begin{barticle}
\bauthor{\bsnm{Carson}, \binits{E.}},
\bauthor{\bsnm{Khan}, \binits{N.}}:
\batitle{Mixed precision iterative refinement with sparse approximate inverse
  preconditioning}.
\bjtitle{SIAM J. on Scientific Computing}
\bvolume{45}(\bissue{3}),
\bfpage{131}--\blpage{153}
(\byear{2023})
\doiurl{10.1137/22M1487709}
\end{barticle}
\endbibitem



\bibitem[\protect\citeauthoryear{Watts-III.}{1981}]{watt:81}
\begin{barticle}
\bauthor{\bsnm{Watts-III.}, \binits{J.W.}}:
\batitle{A conjugate gradient truncated direct method for the iterative
  solution of the reservoir simulation pressure equation}.
\bjtitle{Society of Petroleum Engineers J.}
\bvolume{21},
\bfpage{345}--\blpage{353}
(\byear{1981})
\end{barticle}
\endbibitem

\bibitem[\protect\citeauthoryear{}{2023}]{hsl:2023}
\begin{botherref}
{HSL.} {A} collection of {Fortran} codes for large-scale scientific
  computation.
\url{https://www.hsl.rl.ac.uk}
(accessed 2024)
\end{botherref}
\endbibitem

\end{thebibliography}

\end{document}